\documentclass[reqno]{amsart}
\usepackage[utf8]{inputenc}
\usepackage[margin = 1in]{geometry}
\usepackage{latexsym,amsmath,color,amsthm,epsfig,mathrsfs,mathdots,enumerate}
\usepackage{mathtools}
\mathtoolsset{showonlyrefs}
\usepackage{graphicx}
\usepackage{amssymb}
\usepackage{mdframed}
\usepackage{fancyhdr}
\usepackage{tikz-cd}
\usepackage{verbatim}
\usepackage{nicefrac}
\usepackage{bbm}
\usepackage{hyperref}
\hypersetup{
    colorlinks = true,
    citecolor=black,
    filecolor=black,
    linkcolor=black,
    urlcolor=blue
}
\usepackage{amsfonts, esint}
\usepackage{color}
\usepackage{dsfont}

\numberwithin{equation}{section}

 \usepackage{lastpage}
\newtheorem{thm}{Theorem}[section]
\newtheorem{prop}[thm]{Proposition}
\newtheorem{cor}[thm]{Corollary}

\newtheorem{lemma}[thm]{Lemma}

\newtheorem{preremark}[thm]{Remark}

\newtheorem*{nota}{Notation}
\newenvironment{remark}{\begin{preremark}\rm}{\medskip \end{preremark}}
\numberwithin{equation}{section}

\newcommand{\norm}[1]{\left\Vert#1\right\Vert}

\newcommand{\R}{\mathbb R}
\newcommand{\eps}{\varepsilon}

\newcommand{\grad} {\nabla}

\newcommand{\bdary} {\partial}

\newcommand{\dd} {\; \mathrm{d}}

\newcommand{\diag} {\mathrm{diag}}
\newcommand{\Jap}[1]{\langle #1 \rangle}
\newcommand{\SP}{\mathcal{S}}
\newcommand{\AN}{\mathcal{A}}

\definecolor{sh}{RGB}{255,0,100}

\begin{document}

\begin{abstract}
    We prove a lower bound for the entropy dissipation of the Landau equation with Coulomb potentials by a weighted Lebesgue norm $L^3_{-5/3}$. In particular, we enhance the weight exponent from $-5$, which was established by Desvillettes in \cite{desvillettes2015entropy}, to $-5/3$. Moreover, we prove that the weighted Lebesgue norm $L^3_{-5/3}$ is optimal for both exponents.
\end{abstract}

\title{Entropy dissipation estimates for the Landau equation with Coulomb potentials}
\author{Sehyun Ji}
\address[Sehyun Ji]{Department of Mathematics, University of Chicago,  Chicago, Illinois 60637, USA}
\email{jise0624@uchicago.edu}
\maketitle

\section{Introduction}
 We study the entropy dissipation for the Landau equation with Coulomb potentials. The Landau equation is a kinetic equation that models the evolution of a distribution of particles in a plasma or a diluted gas. The distribution $f(t,x,v) : \R_{\geq 0} \times\R^3 \times \R^3 \to \R_{\geq 0}$ describes the density of particles with respect to their velocity and position at any given time. The Landau equation is given by
\begin{equation} \label{e: Landau}
    \partial_t f(t,x,v)+v\cdot \grad_x f(t,x,v) = Q(f,f)(t,x,v), \quad f(0,x,v)=f_0(x,v).
\end{equation}
The Landau collision operator $Q$ acts on $f(t,x,\cdot)$ for fixed values of $(t,x)$ and is defined as\footnote{In the paper, we omit $(t,x)$ in $f(t,x,v)$ and write $f(v)$ when $(t,x)$ is fixed.}
\begin{equation}
    Q(f,f)(v) := \partial_i  \int_{\R^3} a_{ij}(v-w) \left(
    f(w) \partial_j f(v) - f(v)\partial_j f(w) \right) \dd w,
\end{equation}
where the coefficient matrix $(a_{ij})_{1\leq i,j \leq 3}$ is given by
\begin{equation}
a_{ij}(v):= |v|^{\gamma+2} \left(\delta_{ij}-\frac{v_iv_j}{|v|^2} \right),
\end{equation}
with $\gamma \geq -3$. In the space homogeneous case, the equation does not depend on $x$, hence \eqref{e: Landau} becomes
\begin{equation} \label{e: homogeneous}
    \partial_t f(t,v)=Q(f,f)(t,v),\quad f(0,v)=f_0(v).
\end{equation}
\let\thefootnote\relax
\footnote{2020  Mathematics Subject Classification : 35B45}
\footnote{Keywords: the Landau-Coulomb equation, entropy dissipation, the weighted Lebesgue norm, a priori estimate, the Landau equation with very soft potentials}

The regime $\gamma \in[-3,-2)$ is referred to as very soft potentials. In particular, $\gamma=-3$ is referred to as Coulomb potentials, henceforth the equation \eqref{e: Landau} is known as the Landau-Coulomb equation. 
In the case of very soft potentials, the diffusion matrix $(a_{ij})_{1\leq i,j \leq 3}$ has a singularity at the origin and becomes more singular as $\gamma$ approaches $-3$. 
 The Landau-Coulomb equation has the most singular diffusion among the very soft potentials, and thus is the most interesting case. Hence, we focus on studying the Landau-Coulomb equation. It has been recently proved by Guillen and Silvestre in \cite{guillen2023global} that the Landau-Coulomb equation has a global smooth solution for any reasonable initial data. See e.g. \cite{golse2022partial, golse2022partialarxiv, imbert2021regularity, silvestre2022regularity} for related questions and works on the Landau-Coulomb equation and the Landau equation with very soft potentials.

The entropy dissipation is important in the study of the Landau and Boltzmann equations. For example, it was used to introduce $H$-solutions for the Landau and Boltzmann equations in \cite{villani1998new} by Villani. Another example is the construction of renormalized solutions of the Landau and Boltzmann equations in \cite{alexandre2004landau} by Alexandre and Villani. See \cite{ alexandre2000entropy, chaker2020coercivity, desvillettes2005boltzmann, gressman2011sharp, mouhot2006linear, villani1999regularity} for other works studying entropy dissipation and its application for the Landau and Boltzmann equations.

 In the present paper, we are interested in studying the coercivity of the entropy dissipation of the Landau-Coulomb equation. The entropy dissipation $D(f)$ is given by the following formula 
\begin{align} 
D(f)&:= -\int_{\R^3} Q(f,f)(v) \log f(v) \dd v\\
\label{e: entropy dissipation}
    &=\frac 1 2  \iint_{\R^3\times \R^3} a_{ij}(v-w)  f(v)f(w) \left(\frac{\partial_i f(v)}{f(v)} -\frac{\partial_i f(w)}{f(w)} \right)\left(\frac{\partial_j f(v)}{f(v)} -\frac{\partial_j f(w)}{f(w)} \right) \dd v \dd w
\end{align}
for fixed values of $(t,x)$.
It is clear $D(f)$ is nonnegative from the formula.

We impose assumptions on macroscopic hydrodynamic quantities: mass, energy, and entropy. Our a priori estimate will depend on these quantities. More precisely, we assume there are lower and upper bounds on mass, an upper bound on energy, and an upper bound on entropy for each $(t,x)$:
\begin{gather} 
    \label{e: hydronamic1}
    0< m_0 \leq \int_{\R^3} f(v)  \dd v \leq M_0,\\
    \label{e: hydronamic2}
    \int_{\R^3} f(v) |v|^2 \dd v \leq E_0 ,\\
    \label{e: hydronamic3}
    \int_{\R^3} f(v) \log f(v) \dd v \leq H_0.
\end{gather}

In the space homogeneous setting, it is well know that the mass, momentum, and energy are conserved in time and the entropy is monotone decreasing in time. Therefore, we only need to assume the hydrodynamic bounds \eqref{e: hydronamic1},\eqref{e: hydronamic2}, and \eqref{e: hydronamic3} to hold initially.\\ 

We introduce some notations for our convenience.
\begin{nota}
We use the Japanese bracket convention $\Jap{v} := \sqrt{1+|v|^2}$.
\end{nota}

\begin{nota}
We define the weighted Lebesgue norm $L^p_{-q}$ for $p \geq 1$ and $q\in \R$ for a space $X$ as
\begin{equation*}
    \norm{f}_{L^p_{-q}(X)} :=\left(\int_{X} {|f(v)|^p} \Jap{v}^{-pq} \dd v\right)^{1/p}.
\end{equation*}
\end{nota} 
\noindent

\begin{nota}
For non-negative $Y$ and $Z$, we write $Y\gtrsim Z$ or $Z \lesssim Y$ if $Y\geq CZ$ for some constant $C>0$ that only depends on hydrodynamic bounds $m_0,M_0,E_0,H_0$. In addition, we write $Y\approx Z$ if $Y\gtrsim Z$ and $Z \gtrsim Y$.
\end{nota}

\
In \cite{desvillettes2015entropy}, Desvillettes studied the entropy dissipation for the Landau-Coulomb equation and proved that
\begin{equation} \label{e: Desvillettes }
D(f) +1  \gtrsim \norm{ \grad \sqrt{f}}_{L^2_{-3/2}}^2 +1 \gtrsim \norm{f}_{L^3_{-3}(\R^3)}.
\end{equation}
Note that the second inequality of \eqref{e: Desvillettes } comes from the weighted Sobolev embedding. \\

Our first main result is improving the weight exponent of the last term in \eqref{e: Desvillettes } from $-3$ to $-5/3$, which provides a better asymptotics for large velocities. The precise statement of theorem is as follows.

\begin{thm} \label{thm}
Let $f$ be a solution of \eqref{e: Landau} that satisfies hydrodynamic bounds \eqref{e: hydronamic1},\eqref{e: hydronamic2}, and \eqref{e: hydronamic3}. Then, we have a coercivity estimate
\begin{equation} \label{e: thm}
D(f) \geq c_1 \norm{f}_{L^3_{-5/3}(\R^3)} - 24M_0^2,
\end{equation}
where $c_1>0$ only depends on hydrodynamic bounds $m_0,M_0,E_0,H_0$.
\end{thm}

A simplified proof for the first inequality of \eqref{e: Desvillettes } was provided in \cite[Proposition 2.2]{golse2022partialarxiv}. The lower bound for eigenvalues of the diffusion matrix $A[f]$, which is
\begin{equation} \label{e: low diffusion}
    A[f] (v) := \int_{\R^3} a_{ij}(v-w)f(w) \dd w  \gtrsim \frac{1}{\Jap{v}^3} I,
\end{equation}
plays a crucial role in this proof. We note that our proof of Theorem \ref{thm} is related to this simplified proof.  However, in both the proof of \cite{golse2022partialarxiv} and the proof of Theorem \ref{thm}, the cutoff version of $A[f]$ is used to overcome the singularity of $a_{ij}$ at the origin. The exact definition of the cutoff matrix, which will be called $\widetilde{A}[f]$, is introduced in section 2, but we mention that $\widetilde{A}[f]$ also satisfies the same lower bound as in \eqref{e: low diffusion}.\\

Furthermore, the second main result of this paper is proving the weighted Lebesgue space $L^3_{-5/3}$ is optimal in the following sense.

\begin{thm} \label{thm2}
    Let $p\geq 1$ and $q\in \R$. If there exists a constant $c_2>0$, only depending on hydrodynamic bounds, such that
    \[
    D(h) + 1 \geq c_2 \norm{h}_{L^p_{-q}(\R^3)},
    \]
    for every solution $h$ of \eqref{e: Landau} that satisfies hydrodynamic bounds \eqref{e: hydronamic1},\eqref{e: hydronamic2}, and \eqref{e: hydronamic3}, then $p\leq 3$. Moreover, if $p=3$, then $q\geq 5/3$.
\end{thm}

The weighted Lebesgue norm $L^3_{-5/3}$ that we obtained in Theorem \ref{thm} was predicted in the recent paper \cite{chaker2022entropy} of Chaker and Silvestre. In the paper, they study an entropy dissipation estimate for the Boltzmann equation and suggest that the weight exponent $-3$ in \cite{desvillettes2015entropy} may not be optimal.  Recall that the Boltzmann equation is a kinetic equation with an integro-differential collision operator. The Landau equation, which is the focus of our paper, emerges as a grazing collision limit of the Boltzmann equation. This procedure is similar to recovering the Laplacian operator $-\Delta$ from the fractional Laplacian $(-\Delta)^s$ by sending $s\to 1$; See e.g. \cite{alexandre2002boltzmann, alexandre2004landau, desvillettes1992grazing, goudon1997fokker, villani1998new}.

However, it is worth pointing out that the result of \cite{chaker2022entropy} does not directly imply Theorem \ref{thm}. The proof in \cite{chaker2022entropy} takes advantage of certain non-local features of the Boltzmann collision operator. As a result, the coefficient in the coercivity estimate found in \cite{chaker2022entropy} would degenerate in the grazing collision limit, even if the proper normalization is taken into account.

The optimality of the $L^3$ norm might seem natural because, if we ignore the weight, the Sobolev embedding tells us $\grad \sqrt{f} \in L^2$ implies $\sqrt{f} \in L^6$. However, it is not obvious that the current proof of \eqref{e: Desvillettes } gives the optimal result since it only uses the diffusion matrix $A[f]$ is bounded from below by the identity matrix, again ignoring the weight. It has been recently studied by Cabrera, Gualdani, and Guillen in \cite{cabrera2023diffusion} that the nonlinear diffusion term of the Landau-Coulomb equation has a stronger regularization effect than the Laplace operator. Thus, as soon as we replace $A[f]$ with the weighted identity matrix by applying the inequality \eqref{e: low diffusion}, we are ignoring the difference in regularizing effects between two diffusion operators $A[f]$ and $\frac{1}{\Jap{v}^3} I $. Therefore, there is a possibility that a better estimate for the entropy dissipation exists. 
That is why it is worth to speculate the optimality of the $L^3$ norm.

The optimality of the weight is slightly trickier. At the inequality \eqref{e: Desvillettes }, it seems like the weight exponent $-3$ is coming from the weighted Sobolev embedding, which obviously cannot be improved. Desvillettes also mentions in \cite{desvillettes2015entropy} that the weight $-3$ seems to be the optimal weight. However, the improvement of the weight exponent actually comes from a better analysis of the first inequality of \eqref{e: Desvillettes }, not the second inequality of \eqref{e: Desvillettes }. 
Recall that the lower bound of eigenvalues for $A[f]$ is crucial in proving the first inequality of \eqref{e: Desvillettes }.
In \cite{silvestre2017upper}, Silvestre proved an improved version of  the inequality \eqref{e: low diffusion}, which is
\begin{equation} \label{in:ee}
A[f](v)= (a_{ij}\ast f) (v) \gtrsim \frac {1}{\Jap{v}^3} \left(\frac{v}{|v|}\right)^{\otimes 2}+ \frac 1 {\langle v \rangle} \left(I- \left(\frac{v}{|v|}\right)^{\otimes 2}\right),
\end{equation}
by taking an anisotropic feature of $a_{ij}(v)$ into account. Applying an improved inequality \eqref{in:ee} result in the improvement of the weight. The matrix on the right hand side of \eqref{in:ee} has a determinant $1/\Jap{v}^5$, which suggests that the weight in Theorem \ref{thm} is $1/\Jap{v}^5$. We remark that the inequality \eqref{in:ee} also works for the cutoff matrix $\widetilde{A}[f]$, which is Lemma \ref{lem: eigenvalue}.

\section{A priori estimate for the entropy dissipation}
In this section, we prove Theorem \ref{thm}. The key idea is to analyze an anisotropy of the diffusion matrix by applying an appropriate change of coordinates. We first start by extracting the main coercivity term and an error term from the entropy dissipation formula \eqref{e: entropy dissipation}.

We remove the singularity of $a_{ij}$ at the origin by considering the cutoff matrix $\widetilde{a}$ defined by
\begin{equation} \label{e: def diffusion}
\widetilde{a}_{ij} (v) :=  \eta(|v|) \frac{1}{|v|}\left( \delta_{ij}-\frac{v_iv_j}{|v|^2}\right)=\eta(|v|)\partial_{ij}(|v|).
\end{equation}
The cutoff function $\eta(x) \in C^\infty(\R)$ equal to $1$ in $[1,\infty)$, equal to $|x|^3$ in $[0,1/2]$, and $0\leq \eta' \leq 2.$ Then, 
\[
\partial_i \widetilde{a}_{ij}(v)=2\eta(|v|)\partial_j\left(\frac{1}{|v|} \right) \Rightarrow \partial_{ij}\widetilde{a}_{ij}(v)=-2\eta'(|v|) \frac{1}{|v|^2}\geq -16.
\]
Since $a_{ij} \geq \widetilde{a}_{ij}$, we have
\begin{align} \label{e: expand}
    D(f) &\geq \frac 1 2\iint_{\R^3\times \R^3}  \widetilde{a}_{ij}(v-w)  f(v)f(w) \left(\frac{\partial_i f(v)}{f(v)} -\frac{\partial_i f(w)}{f(w)} \right)\left(\frac{\partial_j f(v)}{f(v)} -\frac{\partial_j f(w)}{f(w)} \right) \dd v \dd w.
\end{align}
Let $\widetilde{A}[f](v):=(\widetilde{a}_{ij}\ast f)(v)$ and we occasionally write $\widetilde{A}$ instead of $\widetilde{A}[f]$ for simplicity. We expand the right hand side of \eqref{e: expand} and use the symmetry of $\widetilde{a}_{ij}$ to get
\begin{align} \label{e: entropy bound}
    D(f)& \geq \int_{\R^3} \widetilde{A}_{ij} (v) \frac{\partial_i f(v)\partial_j f(v)}{f(v)} \dd v -\iint_{\R^3\times \R^3}  \widetilde{a}_{ij}(v-w) \partial_i f(v) \partial_j f(w) \dd v \dd w\\
    \label{e: error}
    & = \int_{\R^3} \widetilde{A}_{ij} (v) \frac{\partial_i f(v)\partial_j f(v)}{f(v)} \dd v + \iint_{\R^3\times \R^3} 
    \partial_{ij}\widetilde{a}_{ij}(v-w)  f(v) f(w) \dd v \dd w.
\end{align}
We integrated by parts twice for the second term of \eqref{e: entropy bound}.
We obtain a bound for the second term of \eqref{e: error},
\begin{align*}
    \iint_{\R^3\times \R^3} 
    \partial_{ij}\widetilde{a}_{ij}(v-w)  f(v) f(w) \dd v \dd w\geq -16M_0^2.
\end{align*}
Therefore,
\begin{equation} \label{e: main}
    D(f)\geq \int_{\R^3} \widetilde{A}_{ij} (v) \frac{\partial_i f(v) \partial_j f(v)}{f(v)} \dd v -16M_0^2,
\end{equation}
and it remains to prove the lower bound for the first term in \eqref{e: error}. Let's state it as a proposition.
\begin{prop} \label{lem: bound}
Let $s(v):=2\sqrt{f(v)}$. Then
\begin{equation}
 \int_{\R^3} \widetilde{A}_{ij} (v)  \frac{\partial_i f(v)\partial_j f(v)}{f(v)} \dd v  +8M_0^2=\int_{\R^3}  \widetilde{A}_{ij}(v) \partial_i s(v) \partial_j s(v)  \dd v +8M_0^2 \gtrsim  \norm{f}_{L^3_{-5/3}(\R^3)} =\frac 1 4 \norm{s}_{L^6_{-5/6}(\R^3)}^2.
\end{equation}
\end{prop}
It is clear that it suffices to show Proposition \ref{lem: bound} to prove Theorem \ref{thm}. Note that the hydrodynamic bound \eqref{e: hydronamic1} on mass become
\begin{equation} \label{e: new mass}
    0<4 m_0 \leq \int_{\R^3} s(v)^2 \dd v \leq 4M_0.
\end{equation}

\begin{remark}
    We briefly explain why the cutoff is necessitated. If we do not remove the singularity of $a_{ij}$, then a simple computation show that $\partial_{ij} a_{ij}$ is equal to a negative constant multiple of $\delta_0$, so the second term in \eqref{e: error} becomes the $L^2$ norm of $f$. It is obvious that the $L^2$ norm of $f$ is not controlled by hydrodynamic quantities.
\end{remark}
A simplified proof of \eqref{e: Desvillettes } provided in \cite{golse2022partialarxiv} essentially comes from combining \eqref{e: main} with the cutoff version of inequality \eqref{e: low diffusion}, that is,
\begin{equation} \label{e: lower diffusion}
    \widetilde{A}[f](v) =(\widetilde{a}_{ij}\ast f) (v) \gtrsim \frac{1}{\Jap{v}^3} I.
\end{equation}

As mentioned in the last paragraph of the introduction, the lower bound \eqref{e: lower diffusion} for eigenvalues of $\widetilde{A}[f]$ can be improved by taking into account that $\widetilde{A}[f]$ is anisotropic. For tangential directions, a lower bound of eigenvalues has a decay rate of $1/\Jap{v}$, not $1/\Jap{v}^3$, as in \eqref{in:ee}. A corresponding statement is given below as Lemma \ref{lem: eigenvalue}. We mostly follow the proof in \cite[Lemma 3.2]{silvestre2017upper}, though a slight modification is needed because we are considering the cutoff matrix $\widetilde{a}_{ij}$, not the original matrix $a_{ij}$.

\begin{lemma} \label{lem: eigenvalue}
The cufoff diffusion matrix $\widetilde{A}[f]$ satisfies the following lower bound  for all non-zero $v\in\R^3$:
\begin{equation} \label{e: eigenvalue}
\widetilde{A}[f] (v) \geq c_0 \left( \frac {1}{\Jap{v}^3} \left(\frac{v}{|v|}\right)^{\otimes 2}+ \frac 1 {\langle v \rangle} \left(I- \left(\frac{v}{|v|}\right)^{\otimes 2}\right) \right)
\end{equation}
for some $c_0=c_0(m_0,M_0,E_0, H_0)>0$.
\end{lemma}

\noindent
To prove Lemma \ref{lem: eigenvalue}, we need the following auxiliary lemma. This lemma was proven in \cite[Lemma 4.6]{silvestre2016new} or \cite[Lemma 3.3]{silvestre2017upper}, but we provide a proof for reader's convenience.
\begin{lemma} \label{lem: concentration}
    There exist positive numbers $R,l,$ and $\mu$, only depending on hydrodynamic bounds $m_0,M_0,E_0,H_0$, such that
    \begin{equation*}
    \{ v\in B_R : f(v) \geq l \} \geq \mu.
    \end{equation*}
\end{lemma}
\begin{proof}
    It is well known that there exists $\widetilde{H}_0=\widetilde{H}_0(M_0,E_0,H_0)$ so that 
    \[
    \int_{\R^3} f(v) \log(1+f(v)) \dd v \leq \widetilde{H}_0.
    \]
    Choose $R,l$, and $\mu$ as follows: 
    \begin{equation*}
        R :=\sqrt{\frac {2E_0}{m_0}}, \quad l: =\frac{m_0}{4|B_R|}, \quad \mu := \frac{m_0}{8\Lambda},
    \end{equation*}
    where $\Lambda>0$ satisfies $\log(1+\Lambda)=8\widetilde{H}_0 m_0^{-1}$. Indeed, we have
    \begin{equation}
        \int_{\R^3 \setminus B_R} f(v) \dd v \leq \frac{E_0}{R^2}=\frac{m_0}{2}, \: \int_{B_R \cap \{f < l \}} f(v) \dd v \leq l |B_R|=\frac{m_0}{4}, \: \int_{  \{f >\Lambda \}} f(v) \dd v \leq \frac{\widetilde{H}_0}{ \log(1+\Lambda)}=\frac{m_0}{8}.
    \end{equation}
Hence, we obtain
    \begin{equation*}
       \frac{m_0}{8} \leq  \int_{B_R \cap \{ l \leq f \leq \Lambda \} } f(v) \dd v \leq \Lambda| B_R \cap \{  l\leq f \}|.
    \end{equation*}
\end{proof}

\begin{proof}[Proof of Lemma \ref{lem: eigenvalue}]
Let $e$ be a unit vector. It suffices to show
\begin{equation*}
    \langle \widetilde{A}_{ij}(v)e, e \rangle \gtrsim
    \begin{cases}
         \Jap{v}^{-1} & v\cdot e =0,\\
         \Jap{v}^{-3} & v \parallel e.
    \end{cases}
\end{equation*}
The definition \eqref{e: def diffusion} for $\widetilde{a}$ gives us
\begin{align*}
    \langle \widetilde{A}_{ij}(v)e, e \rangle
    & =\int_{\R^3} \frac{\eta(|w|)}{|w|} \left(1- \left(\frac{w\cdot e}{|w|} \right)^2 \right) f(v-w) \dd w\\
    & \geq l \int_{  \{f(v-w)\geq l \} \cap B_R(v)   } \frac{\eta(|w|)}{|w|}\left(1- \left(\frac{w\cdot e}{|w|} \right)^2 \right)  \dd w.
\end{align*}
Suppose that there is $\eps>0$ such that
\begin{equation} \label{e: eps}
    \left| \{ w\in B_{R}(v) :  1- \left(\frac{w\cdot e}{|w|} \right)^2 < \eps   \}  \right| <\frac{\mu}{2}.
\end{equation}
By Lemma \ref{lem: concentration}, at least one of the sets 
\[
 \left| \{ w\in B_{R}(v) \cap B_{1/2} :  f(v-w)\geq l \text{ and } 1- \left(\frac{w\cdot e}{|w|} \right)^2 \geq \eps \}  \right|
\]
or
\[
 \left| \{ w\in B_{R}(v) \cap B_{1/2}^c :  f(v-w)\geq l \text{ and } 1- \left(\frac{w\cdot e}{|w|} \right)^2 \geq \eps  \}  \right| 
\]
has measure greater than $\mu/4$. Recall that $\eta(s)=|s|^3$ for $s\in [0,1/2]$ and $\eta(s) \geq 1/8$ for $s \in [1/2,\infty)$.
Depending on the cases,
\begin{equation} \label{e: anistropy}
    \langle \widetilde{A}_{ij}(v)e, e \rangle \geq   \frac{l\mu \eps}{4}  \min \left(\frac{1}{8(|v|+R)}, \int_{B_{r_0}} |w|^2 \dd w\right) \gtrsim  l\mu \eps \cdot \min \left( \frac{1}{|v|+R} , \mu^{5/3} \right)
\end{equation}
where $r_0$ satisfies $|B_{r_0}|=\mu/4$. \\
If $|v|<\max(2R,\mu^{-5/3} )$, then there exists a uniform $\eps$ so that \eqref{e: eps} is satisfied. Thus, the inequality \eqref{e: anistropy} gives the lower bound of eigenvalues. Therefore, we can assume without loss of generality that $|v| \geq \max(2R,\mu^{-5/3} )$. Then, the inequality \eqref{e: anistropy} implies
\begin{equation} \label{e: lower}
     \langle \widetilde{A}_{ij}(v)e, e \rangle \gtrsim \frac{2l \mu \eps}{3\Jap{v}} .
\end{equation}
First, suppose that $v\perp e$. Since $|v| \geq 2R$ and $w \in B_R(v)$, 
\[
\left(\frac{w\cdot e}{|w|} \right)^2 \leq \frac{1}{4}.
\]
We choose $\eps=3/4$, then \eqref{e: eps} is satisfied. Therefore, \eqref{e: lower} implies
\begin{equation*}
     \langle \widetilde{A}_{ij}(v)e, e \rangle \gtrsim \frac{ l \mu }{2\Jap{v}} .
\end{equation*}
Next, suppose that $v \parallel e$. The set
\[
\{ w\in \R^3 : 1- \left(\frac{w\cdot e}{|w|} \right)^2 < \eps   \}
\]
is a cone at the origin with an opening angle $\approx \sqrt{\eps}$. Its intersection with $B_R(v)$ is contained in some cylinder with height $2R$ and radius $\approx \sqrt{\eps} |v|$. See the cylinder of Figure \ref{fig:example1}.

\begin{figure}[h]
    \centering
    \includegraphics[width=0.7\linewidth]{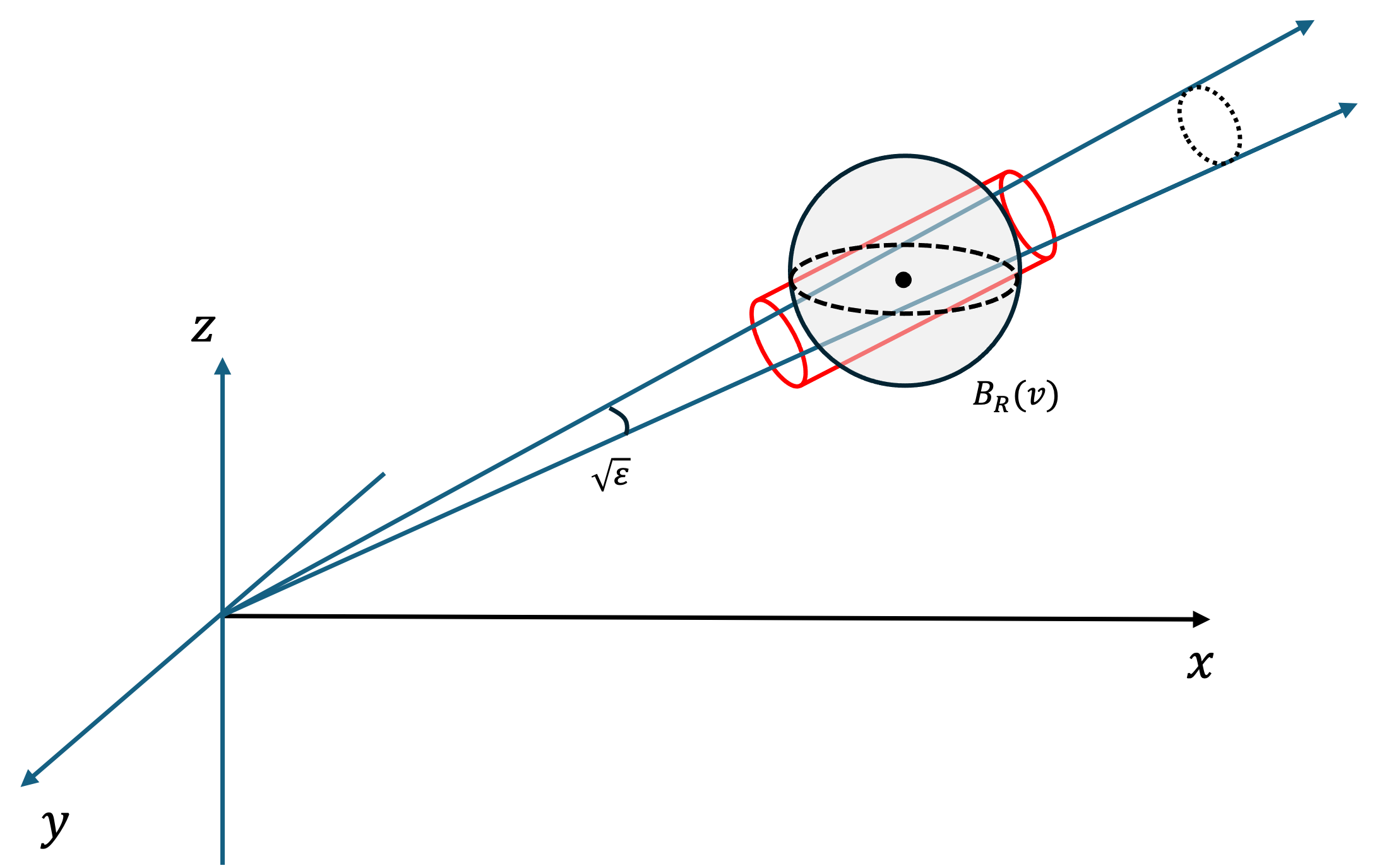}
    \caption{A cone with an opening angle $\sqrt{\eps}$ and $B_R(v)$}
    \label{fig:example1}
\end{figure}
Then,
\begin{equation*} 
    \left| \{ w\in B_{R}(v) :  1- \left(\frac{w\cdot e}{|w|} \right)^2 < \eps   \}  \right| \lesssim R ({\sqrt{\eps}|v|} )^2 = R\eps |v|^2
\end{equation*}
and \eqref{e: eps} is achieved by choosing $\eps \approx \mu(2R)^{-1}|v|^{-2} $.
As a consequence, \eqref{e: lower} implies
\begin{equation*}
     \langle \widetilde{A}_{ij}(v)e, e \rangle \gtrsim \frac{l \mu^2}{2R\Jap{v}^3} .
\end{equation*}
\end{proof}

\begin{remark} \label{rem: mass}
    We introduced the constant $c_0$ rather than using $\gtrsim$ notation in Lemma \ref{lem: eigenvalue} to precisely compute the error term in \eqref{e: thm}. Without loss of generality, we can always make $c_0$ as small as we want at the statement of Lemma \ref{lem: eigenvalue}. Choosing a smaller $c_0$ gives a worse coefficient for the coercivity term, but provides a better error estimate. In this manuscript, we choose $c_0$ to be less than or equal to $2M_0$.\\
    On the other hand, it is also possible to show directly $c_0 \le 2M_0$. Indeed, since $\eta(|v|)/|v|$ is bounded above by $2$, we have $\widetilde{a}_{ij}(v) \le 2 I$. This implies
    \[
    \widetilde{A}[f](v)=\int_{\R^3} f(v-w)\widetilde{a}_{ij}(w) \dd w \le 2 \left(\int_{R^3}f(v-w) \dd w \right) I \le 2M_0 I.
    \]
    Choosing $v$ arbitrarily close to the origin in Lemma \ref{lem: eigenvalue}, we have $c_0 \le 2M_0.$
\end{remark}

\noindent

 We partition $\R^3-\{0\}$ into annuli $\AN_N:=\{v \in \R^3 : N-1 < |v| \le N \}$ where $N$ is a positive integer. Fix $N$ and let's consider a point $v_0 \in \AN_N$ such that $|v_0|=N.$  We consider a spherical cap $\SP_{v_0} $ of $\partial B_N(0)$ centered at $v_0$ with a radius $1$. That is,
    \begin{equation}
        \SP_{v_0} := \partial B_N(0) \cap B_1(v_0)= \{v \in \R^3 : |v|=N, \quad |v-v_0|\le 1 \}.
    \end{equation}
Connecting $\SP_{v_0}$ with the origin determines the cone $\mathcal{C}_{v_0}$ centered at the origin. Define a truncated spherical cone $F_{v_0}$ as
\begin{equation}
    F_{v_0} := \mathcal{C}_{v_0} \cap \AN_N.
\end{equation}
We often write $F_0$ instead of $F_{v_0}$ for the sake of simplicity. 

\begin{figure}[h]
    \centering
    \includegraphics[width=0.6\linewidth]{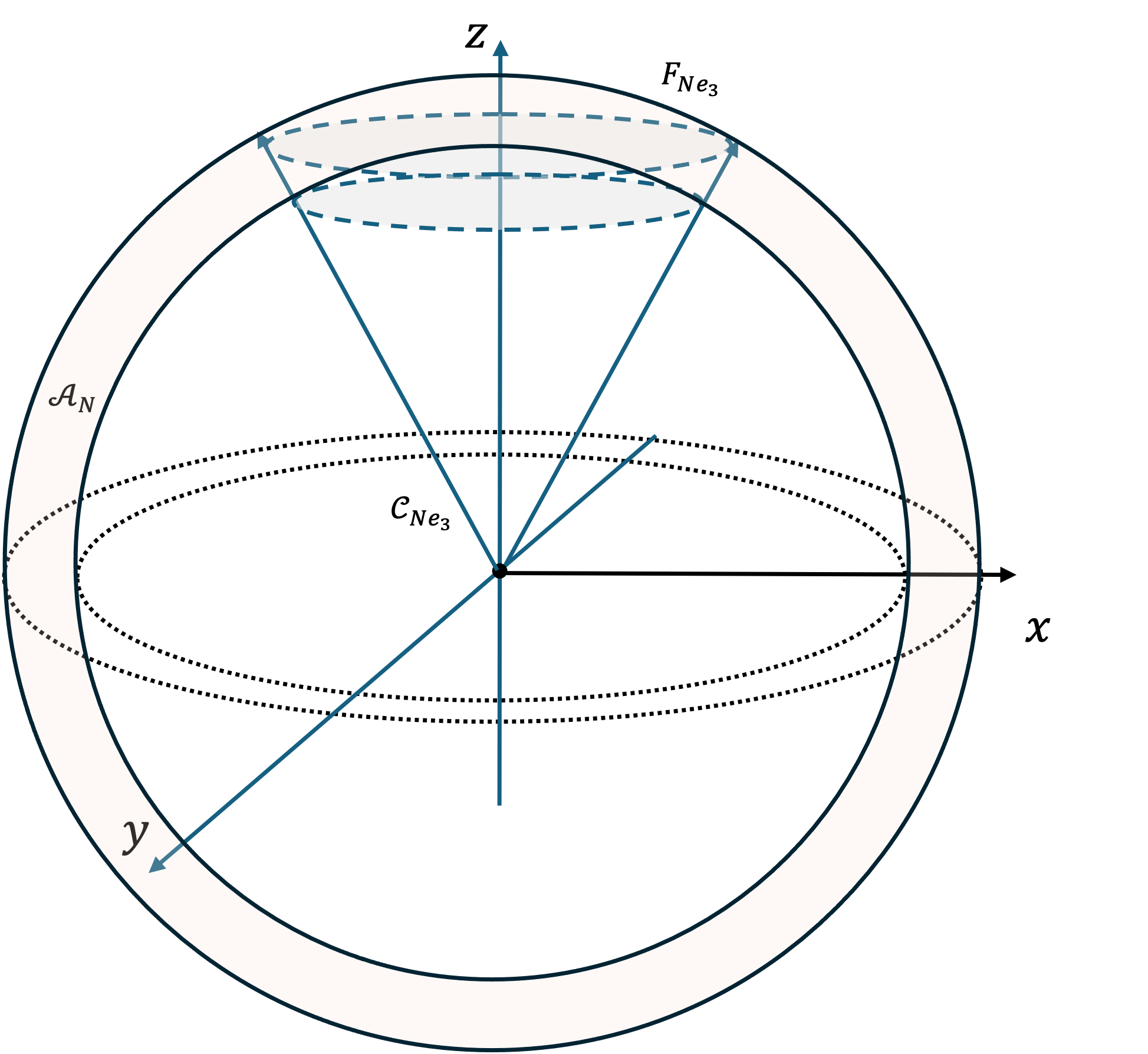}
    \caption{$\AN_N, \mathcal{C}_{Ne_3}$, and $F_{Ne_3}$}
    \label{fig:example2}
\end{figure}

We prove the following covering lemma, inspired by the Vitali covering lemma.
 \begin{lemma}[Covering lemma] \label{lem: covering}
    For every $N$, there exist points $v_1,\cdots ,v_{c(N)}$ so that $\SP_{v_1},\cdots,\SP_{v_{c(N)}}$ cover $\bdary B_N(0)$, with no points of $\bdary B_N(0)$ being covered more than $16$ times. In particular, $F_{v_1},\cdots,F_{v_{c(N)}}$ cover $\AN_N$, with no points of $\AN_N$ being covered more than $16$ times.
 \end{lemma}
 \begin{proof}
     Let $\theta$ be the opening of the cone $\mathcal{C}_{v_0}$, then we have
     \[
     N\sin\left(\frac{\theta}{4}\right)=\frac{1}{2}.
     \]
     We start by defining a spherical cap $\tilde{\SP}_{v_0}$ of $\bdary B_N(0)$ with a different size:
     \[
     \tilde{\SP}_{v_0} := \bdary B_N(0) \cap \{v \in \R^3 : |v-v_0|\le \tau_1 \},
     \]
     where
     \[
     N\sin\left(\frac{\theta}{12}\right)=\frac{\tau_1}{2}.
     \]
        Note that $\tau_1$ is chosen so that the opening of the corresponding cone $\tilde{\mathcal{C}}_{v_0}$ of $\tilde{\SP}_{v_0}$ is $\frac{\theta}{3}$.\\
     Since $\bdary B_N(0)$ is compact, we can find a finite collection of spherical caps that $\{\tilde{\SP}_{j} \}_j$ covers $\bdary B_N(0)$. We choose an arbitrary cap from the collection and call it $\tilde{\SP}_{i_1}$. Inductively, assume $\tilde{\SP}_{i_1}, ,\cdots \tilde{\SP}_{i_k}$ are chosen. If there is a spherical cap that is disjoint from $\tilde{\SP}_{i_1} \cup \cdots \cup \tilde{\SP}_{i_k}$, let $\tilde{\SP}_{i_{k+1}}$ be such spherical cap (If there are more than one spherical cap satisfies this condition, pick one arbitrarily). Otherwise, we set $c(N):=k$ and terminate the recursive definition. Relabel the spherical caps as $ \tilde{\SP}_{v_1},\cdots, \tilde{\SP}_{v_{c(N)}}$. We claim that $\SP_{v_1},\cdots,\SP_{v_{c(N)}}$ form a covering of $\bdary B_N(0)$ that satisfies the statement of the lemma.\\
     It is enough to show that each point $x \in \bdary B_N(0)$ is covered at least once and at most $16$ times by the spherical covers $\SP_{v_1},\cdots,\SP_{v_{c(N)}}$. Since $\{\tilde{\SP}_{j} \}_j$ is a covering, there exists $j$ such that $x\in \tilde{\SP}_j$. By the inductive definition, there is $i\in \{1,\cdots,c(N)\}$ such that $\tilde{\SP}_{v_i}$ intersects $\tilde{\SP}_j$. It is easy to see that $x\in \tilde{\SP}_j \subset \SP_{v_i}$, so $x$ is covered at least once.\\
     Now, suppose that $x$ is covered exactly $m$ times by the covering $\{\SP_{v_1},\cdots,\SP_{v_{c(N)}} \}$. Without loss of generality, suppose that $x\in \SP_{v_i}$ for $i=1,2,\cdots,m$. It follows that
     \[
     v_1, \cdots, v_m \in \SP_x.
     \]
     Then,
     \[
      \overline{\SP}_x := \bdary B_N(0)\cap \{v \in \R^3 : |v-x|\le \tau_2 \} \supset \tilde{\SP}_{v_1},\cdots, \tilde{\SP}_{v_m } \
     \]
     where
     \[
     N\sin\left(\frac{\theta}{3}\right)=\frac{\tau_2}{2}.
     \]
        Note that $\tau_2$ is chosen so that the opening of the corresponding cone $\mathcal{C}$ of $\tilde{\SP}$ is $\frac{4\theta}{3}$. The surface area of $\tilde{\SP}_{v_i}$ is
        \[
        2\pi N^2 \left(1-\cos \left(\frac{\theta}{6}\right) \right)=4\pi N^2 \sin^2 \left(\frac{\theta}{12}\right)=\pi \tau_1^2.
        \]
        Similarly, the surface area of $\tilde{\SP}$ is $\pi \tau_2^2$. Since $\tilde{\SP}_{v_1},\cdots, \tilde{\SP}_{v_m }$ are disjoint subsets of $\overline{S}_x$, we have
        \[
        m \pi \tau_1^2 \le \tau_2^2.
        \]
        Since $\frac{\sin{x}}{x}$ is a decreasing function for $0<x<\pi$, we conclude $m \le 16.$
 \end{proof}
\begin{remark} 
    Each spherical cap $\SP_{v_i}$ has a surface area $\pi$. Since the total surface area of $\bdary B_N(0)$ is $4\pi N^2$, we deduce that $c(N) \approx N^2$. However, the number of truncated spherical cone covers of Lemma \ref{lem: covering} does not affect the proof of Proposition \ref{lem: bound}.
\end{remark}

\begin{remark}
    It is possible to work with conical frustums instead of truncated spherical cones. The geometry of a conical frustum is slightly simpler, as its base and top are flat and parallel. However, conical frustums do not provide an exact partition of the annulus $\AN_N$, so we need to ensure that the overlaps occur a finite number of times independent of $N$.
\end{remark}
\begin{remark}
    The previous version of this manuscript had an issue with the covering lemma. The linear map $L_N:v\to Nv$ does not send an ellipsoid to a sphere unless we shift the map $L_N$ to be centered at $v_0$. However, then $L_N$ depends on the center of the ellipsoid so it does not give a uniform way to send each ellipsoid to a sphere simultaneously. 
\end{remark}

In \cite[Lemma 4.1]{cameron2018global}, Cameron, Silvestre, and Snelson introduce a change of variables to capture the anisotropy of the diffusion.
A similar change of coordinates was also defined in  \cite{imbert2022global}.
Inspired by their definition, we define a linear transformation $T$ as
\begin{equation*}
Te :=\begin{cases}
e & e \perp v_0,\\
Ne & e \parallel v_0.
\end{cases}
\end{equation*}
Note that $T$ is symmetric, i.e., $T=T^t$. \\

Our strategy is to perform a change of coordinates $T$ near $v_0$ to remove the anisotropy before applying the extension theorem for Sobolev spaces. Subsequently, we apply the covering lemma \ref{lem: covering} to cover the whole $\AN_N$ with truncated spherical cones $F_{v_0}$, not only near $v_0$.\\

We begin by recalling a classic result from \cite[Chapter VI]{stein1970book} by Stein concerning the Sobolev extension for a Lipschitz domain $\Omega \subset \R^3$. Stein works with a class of Lipschitz domains, known as \textit{minimally smooth} domains, whose boundaries satisfy the following conditions: \\
There exist an $\eps >0$, an integer $J$, a Lipschitz constant $L>0$, and a countable collection of balls $\{B_\delta(x_j)\}_{j=1,2,\cdots}$ so that
\begin{enumerate}[i.]
    \item No point of $\R^3$ is contained more than $J$ distinct balls.
    \item An $(\delta/2)-$neighborhood of $\bdary \Omega$ is covered by the balls. 
    \item For each $j$, there exists a Lipschitz function $L_j:\R^2 \to \R$ with a Lipschitz constant $L$ such that $ B_\delta(x_j) \cap \Omega$ can be translated and rotated to coincide with the intersection of $B_\delta(0)$ and the supergraph of $L_j$.
\end{enumerate}
Note that the minimally smoothness condition is sometimes referred to as the strong local Lipschitz condition(See \cite{adams2003book}). 
\begin{thm}[From \cite{stein1970book}] \label{thm: extension}
    Let $\Omega$ be a minimally smooth domain in $\R^3$. Then, there exists a bounded extension operator
    \begin{equation}
        P_\Omega:H^1(\Omega) \to H^1(\R^3)
    \end{equation}
    such that for all $f\in H^1(\Omega)$,
    \begin{equation}
        P_\Omega f\vert_\Omega =f.
    \end{equation}
    The operator norm of $P_\Omega$ only depends on the constants of the minimally smooth domain $\Omega$.
\end{thm}
A similar result was achieved by Calderon \cite{calderon1961book}. See also \cite{jones1981sobolev,luke2006sobolev} for stronger versions of the Sobolev extension theorem.\\

Let $\Omega_0$ be the image of $F_0$ under $T$. That is,
\[
\Omega_0:= T(F_0) =T \mathcal{C}_{v_0} \cap T \AN_N
\]
In other words, $\Omega_0$ is an intersection of a cone and an ellipsoidal annulus. Thus, $\Omega_0$'s surface consists of the lateral surface and two ellipsoidal caps. Moreover, since $\Omega_0$ is a surface of revolution, it is easier to analyze the geometry of its cross-section. See Figure \ref{fig:example3}.

\begin{figure}[h]
    \centering
    \includegraphics[width=0.8\linewidth]{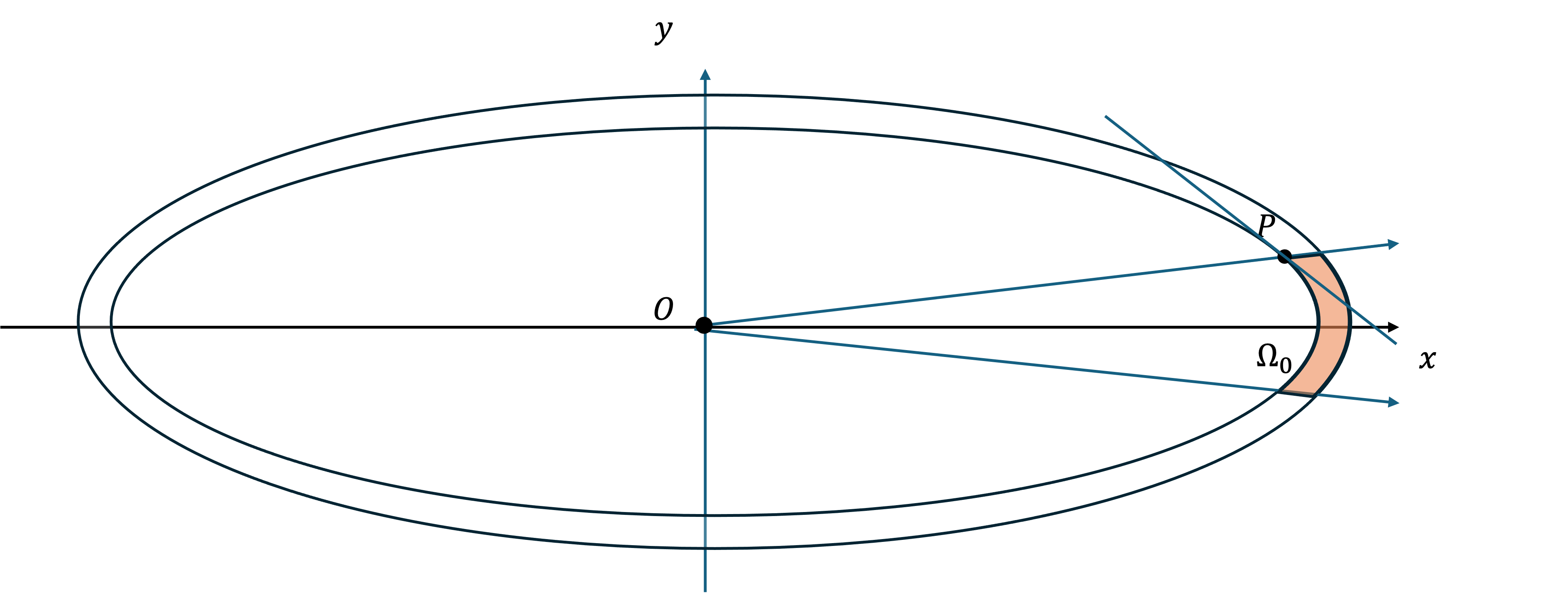}
    \caption{$\Omega_0$ and a point $P$}
    \label{fig:example3}
\end{figure}
Assume $N>1$ without loss of generality.
It is clear that the outer ellipsoid cap has a curvature of at most $1$ at every point. Moreover, the outer ellipsoid cap meets the lateral surface at an obtuse angle. Similarly, the inner ellipsoid cap has a curvature of at most $\frac{N}{N-1} \le 2$ at every point.
Let's verify that the angle between the inner ellipsoid cap and the lateral surface does not degenerate as $N\to \infty$. Without loss of generality, assume $v_0=Ne_1$. Let $P$ be a point on the intersection of the two faces, as shown in Figure \ref{fig:example3}. Recall $\theta$ is the opening of the cone $\mathcal{C}_{Ne_1}$, then
\[
P= \left(N(N-1)\cos\left(\frac{\theta}{2}\right), (N-1) \sin \left(\frac{\theta}{2}\right),0 \right).
\]
The $x$-axis and the line $OP$ meet at an angle less than $\theta$, where $O$ is the origin. The slope of the tangent line at $P$ is
\[
\frac{dy}{dx}=-\frac{N(N-1)\cos\left(\frac{\theta}{2}\right) }{(N-1) \sin \left(\frac{\theta}{2}\right) } \cdot \frac{(N-1)^2}{(N-1)^4}= -\frac{N\cos \left(\frac{\theta}{2}\right)}{(N-1)^2  \sin \left(\frac{\theta}{2}\right) }=-\frac{N^2}{(N-1)^2}\frac{1-\frac{1}{2N^2}}{\sqrt{1-\frac{1}{4N^2}}} \le -1.
\]

It implies that the inner ellipsoid cap meets the lateral surface with an angle that is bounded away from $0$ for every $N$.
We deduce that $\Omega_0$ satisfies the minimally smoothness condition and the constants associated with this condition do not depend on $N$. Therefore, we conclude
\begin{equation} \label{fact}
    \norm{P_{\Omega_0}} \approx 1,
\end{equation}
by Theorem \ref{thm: extension}.
\begin{lemma} \label{lem: ellipsoid}
Let $N$ be fixed. For a point $v_0 \in \AN_N$ such that $|v_0|=N$, we have
\begin{equation*}
   \int_{F_0} \widetilde{A}_{ij}(v) \partial_i s(v) \partial_j s(v) \dd v  + \frac{2M_0}{N} \int_{F_0} |s(v)|^2 \dd v \gtrsim \frac{c_0}{8N^{5/3}}\norm{s}_{L^6(F_0)}^2.
\end{equation*}

\end{lemma}
\begin{proof}
  
Apply the change of coordinates by $T$.  Moreover, let $u=Tv$ and $\widetilde{s}(u)=s(v)$. 
Then,
\begin{align*}
    \int_{F_0} \widetilde{A}_{ij}(v) \partial_i s(v) \partial_j s(v) \dd v 
    &= \int_{\Omega_0} \widetilde{A}_{ij}(v) T_{ki} \partial_k \widetilde{s}(u)  T_{lj} \partial_l \widetilde{s}(u) \frac{1}{|\det T| } \dd u\\
     &=  \frac{1}{N}  \int_{\Omega_0}   T_{ki}\widetilde{A}_{ij}(v)  T_{lj}  \partial_k \widetilde{s}(u) \partial_l \widetilde{s}(u)\dd u\\
      &=  \frac{1}{N}  \int_{\Omega_0}   b_{kl}  \partial_k \widetilde{s}(u) \partial_l \widetilde{s}(u)\dd u,
\end{align*}
where the matrix $b:=T \widetilde{A} T^t$. Note that $b$ also equals to $T^t \widetilde{A} T$ since $T$ is symmetric. \\
Suppose $e$ is a unit vector. Then, 
\begin{equation*}
    \Jap{be,e}=\Jap{\widetilde{A}Te, Te}=
        \begin{cases}
            \Jap{\widetilde{A}e,e} & e \perp v_0,\\
            N^2\Jap{\widetilde{A}e,e} & e \parallel v_0.
        \end{cases}
\end{equation*}
From Lemma \ref{lem: eigenvalue}, it follows that
\begin{equation*}
    b\geq c_0 \frac{N^2}{(N+1)^3} I.
\end{equation*}
Therefore, 
\begin{align} \label{eq: covering} 
    \int_{F_0} \widetilde{A}_{ij}(v) \partial_i s(v) \partial_j s(v) \dd v 
     & \geq c_0 \frac{N}{(N+1)^3} \int_{\Omega_0} |\grad \widetilde{s}(u)|^2 \dd u.
\end{align}

Recall the existence of an extension operator $P_{\Omega_0}$ from Theorem \ref{thm: extension}.
By the Sobolev embedding $H^1 \hookrightarrow L^6$ in 3D,
\begin{equation} \label{e: Sobolev}
     \int_{\Omega_0} \left(|\grad \widetilde{s}(u)|^2 +  \widetilde{s}(u)^2  \right) \dd u = \norm{\widetilde{s}}_{H^1(\Omega_0)}^2  \ge \norm{P_{\Omega_0}}^{-2} \norm{P\widetilde{s}}_{H^1(\R^3)}^2 \gtrsim  \norm{P_{\Omega_0}}^{-2} \norm{P\widetilde{s}}_{L^6(\R^3)}^2  \geq \norm{P_{\Omega_0}}^{-2} \norm{\widetilde{s}}_{L^6(\Omega_0)}^2
\end{equation}
where $\norm{P_{\Omega_0}}$ is the operator norm of $P_{\Omega_0}$. By \eqref{fact}, we can absorb $\norm{P_{\Omega_0}}$ into $\gtrsim$, then
\[
 \int_{\Omega_0} \left(|\grad \widetilde{s}(u)|^2 +  \widetilde{s}(u)^2  \right) \dd u  \gtrsim \norm{\widetilde{s}}_{L^6(\Omega_0)}^2.
\]
Combining \eqref{eq: covering} with
\begin{equation} \label{e: change 1}
    \int_{\Omega_0}  \widetilde{s}(u)^2 \dd u=  N \int_{F_0} s(v)^2 \dd v,
\end{equation}
\begin{equation} \label{e: change 2}
    \norm{\widetilde{s}}_{L^6(\Omega_0)}^2=N^{1/3} \left( \int_{F_0
    } s(v)^6 \dd v \right)^{1/3},
\end{equation}
we obtain
\begin{equation} \label{e: estimate}
    \int_{F_0} \widetilde{A}_{ij}(v) \partial_i s(v) \partial_j s(v) \dd v + c_0\frac{N^2}{(N+1)^3}\int_{F_0}   s(v)^2 \dd v \gtrsim c_0\frac{N^{4/3}}{(N+1)^3}  \norm{s}_{L^6(F_0)}^2.
\end{equation}
We know that $c_0\le 2M_0$ from Remark \ref{rem: mass}, therefore we get
\begin{equation*}
   \int_{F_0} \widetilde{A}_{ij}(v) \partial_i s(v) \partial_j s(v) \dd v  + \frac{2M_0}{N} \int_{F_0} |s(v)|^2 \dd v \gtrsim \frac{c_0}{8N^{5/3}} \norm{s}_{L^6(F_0)}^2.
\end{equation*}
  
\end{proof}

\noindent
As a consequence of Lemma \ref{lem: ellipsoid} and Lemma \ref{lem: covering}, we have the following corollary.
\begin{cor} 
 For each $N$, we have
\begin{equation} \label{e: local estimate}
    \int_{\AN_N} \widetilde{A}_{ij}(v) \partial_i s(v) \partial_j s(v) \dd v+ \frac{2M_0}{N}\int_{\AN_N} |s(v)|^2 \dd v \gtrsim \frac{1}{N^{5/3}} \norm{s}_{L^6(\AN_N)}^2.
\end{equation}
\end{cor}
\begin{proof}
 By Lemma \ref{lem: ellipsoid},
 We have the covering $\{F_1, F_2,\cdots ,F_{c(N)}\}$ of $\AN_N$ by Lemma \ref{lem: covering}. Apply Lemma \ref{lem: ellipsoid} for each cover $F_i$ and add them all for $i=1,\cdots, c(N)$. Then, 
 \begin{equation*} 
    16\int_{ \AN_N } \widetilde{A}_{ij}(v) \partial_i s (v) \partial_j s(v) \dd v  + \frac{32M_0}{N} \int_{\AN_N} |s(v)|^2 \dd v \gtrsim \frac{c_0}{8N^{5/3}} \sum_{i=1}^{c(N)}  \left( \int_{F_i} |s(v)|^6 \dd v \right)^{1/3}.
\end{equation*}
We conclude the proof since
\begin{equation*}
    \frac{1}{N^{5/3}} \sum_{i=1}^{c(N)}  \left( \int_{F_{v_i}} |s(v)|^6 \dd v \right)^{1/3} \geq  \frac{1}{N^{5/3}} \left( \sum_{i=1}^{c(N)}   \int_{F_{v_i}} |s(v)|^6 \dd v \right)^{1/3} \geq \frac{1}{N^{5/3}} \left( \int_{\AN_N} |s(v)|^6 \dd v \right)^{1/3}.
\end{equation*}
\end{proof}
\noindent
We are ready to prove Proposition \ref{lem: bound} and consequently our main Theorem \ref{thm}.
\begin{proof}[Proofs of Proposition \ref{lem: bound} and Theorem \ref{thm}]
Summing up \eqref{e: local estimate} for all $N$, we obtain
\begin{align*} \label{e: proof}
        \sum_{N=1}^\infty \frac{1}{N^{5/3}} \norm{s}_{L^6(\AN_N)}^2 
     & \lesssim \sum_{N=1}^\infty \left( \int_{\AN_N}\widetilde{A}_{ij}(v) \partial_i s(v) \partial_j s(v) \dd v  + \frac{2M_0}{N} \int_{\AN_N} |s(v)|^2 \dd v\right)\\
     & = \int_{\R^3}\widetilde{A}_{ij}(v) \partial_i s(v) \partial_j s(v) \dd v  + 2M_0 \sum_{N=1}^\infty \frac{1}{N} \int_{\AN_N} |s(v)|^2 \dd v\\
     & \leq \int_{\R^3}\widetilde{A}_{ij}(v) \partial_i s(v) \partial_j s(v) \dd v  +8 M_0^2.
\end{align*}
In the third inequality, we used \eqref{e: new mass} and the crude estimate $1/N \leq 1$. On the other hand, 
\begin{equation}
    \sum_{N=1}^\infty \frac{1}{N^{5/3}} \norm{s}_{L^6(\AN_N)}^2
    \gtrsim   \sum_{N=1}^\infty   \left( \int_{\AN_N} \frac{|s(v)|^6}{\Jap{v}^5} \dd v \right)^{1/3} 
    \geq \left( \sum_{N=1}^\infty   \int_{\AN_N} \frac{|s(v)|^6}{\Jap{v}^5} \dd v \right)^{1/3}
    =   \left( \int_{\R^3 } \frac{|s(v)|^6}{\Jap{v}^5} \dd v \right)^{1/3}.
\end{equation} 
Bootstrapping the inequalities above, we proved Proposition \ref{lem: bound},
\begin{equation*}
    \int_{\R^3} \widetilde{A}_{ij} (v) \partial_i s(v) \partial_j s(v) \dd v +8M_0^2 \gtrsim   \left( \int_{\R^3} \frac{s(v)^6}{\Jap{v}^5} \dd v \right)^{1/3}.
\end{equation*}
From \eqref{e: main}, Theorem \ref{thm} follows since
\begin{equation}
    D(f)+24M_0^2 \gtrsim \left( \int_{\R^3} \frac{s(v)^6}{\Jap{v}^5} \dd v\right)^{1/3} =4 \norm{f}_{L^3_{-5/3}}.
\end{equation}
\end{proof}

\section{An optimality of the weighted Lebesgue space $L^3_{-5/3}$}
In the previous section, we found that the entropy dissipation is bounded below in terms of $L^3_{-5/3}$ norm. In this section, we prove the weighted Lebesgue norm $L^3_{-5/3}$  is optimal as stated in Theorem \ref{thm2}.\\
Let $f(v)$ denote the standard Gaussian function, i.e.,
\[
f(v) : =\frac{1}{(2\pi)^{3/2}} e^{-|v|^2/2}.
\]
Choose a radial bump function $\varphi \in C^\infty_c(\R^3)$ supported on $B_1$ with mass $1$, i.e., $\norm{\varphi}_{L^1(\R^3)}=1$. More precisely,
\begin{align}
\varphi(v) :=
\begin{cases} \alpha_0 \exp\left( -\frac{1}{1-|v|^2} \right) & |v| < 1,\\
0 & |v|\geq 1,
\end{cases}
\end{align}
where $\alpha_0$ ensures that mass is $1$. Let $N\geq 2$ be a positive integer and $B$ be a positive number. We will determine $B,N$ later and only assume $B\ge N^6$ now. Define $g: \R^3 \to \R $ by
\begin{equation}
g(v):=cB^3N \varphi(BN(v_1-N),Bv_2, Bv_3),
\end{equation}
then the support of $g$ is $E:=\{(x,y,z): N^2(x-N)^2+y^2+z^2 \le 1 \}$. Note that $g(v)$ is scaled in the way to capture an anistropy of the diffusion matrix.
We choose
\[
c:=\frac{1}{BN^5},
\]
then it follows that $c \leq N^{-2} \leq N^6 \leq B.$

Our strategy is to compute the entropy dissipation of $h:=\max(f,g)$. The function $h$ is essentially a standard Gaussian function $f$ with the error $g$. Although, we use $\max(f,g)$ instead of $f+g$ for technical reasons.

Define the sets $E_f:= \{ v \mid  f(v) >g(v)\}$ and $E_g:=\{ v  \mid f(v) < g(v)\}$. 
Since the support of $g$ is $E$, we have $E_g \subset E$. 
We also define the change of variables as $\tilde{v}=(BN(v_1-N),Bv_2, Bv_3)$ and $\tilde{w}=(BN(w_1-N),Bw_2, Bw_3)$, so that $g(v)=cB^3N \varphi(\tilde{v})$ and $g(w)=cB^3N \varphi(\tilde{w})$.

\subsection{Hydrodynamic quantities of $h$}

We check that the hydrodynamic quantities of $h$ are bounded.
\begin{prop} \label{prop: hydrodynamic}
Hydrodynamic quantities of $h(v)$ are bounded as follows:
\begin{gather} 
    \label{e: hydrodynamic4}
     1 \leq \int_{\R^3} h(v)  \dd v \leq 9,\\  
    \label{e: hydrodynamic5}
    \int_{\R^3} h(v) |v|^2 \dd v \leq 9 ,\\
    \label{e: hydrodynamic6}
   \int_{\R^3} h(v) \log h(v) \leq 10+\beta_0,
\end{gather}
where $\beta_0$ is a constant defined as $\int_{\R^3} \varphi(v)|\log \varphi(v) | \dd v$.
\end{prop}
\begin{proof}
We compute the mass and energy of $h$. Recall
\[
\int_{\R^3} f(v) \dd v =1, \int_{\R^3} f(v)|v|^2 \dd v =3. 
\]
On the other hand, by a straight forward computation,
\[
\int_{\R^3} g(v)\dd v =c\int_{\R^3} \varphi(\tilde{v}) \dd\tilde{v} =c\leq 1,
\]
\[
\int_{\R^3} g(v)|v|^2 \dd v =\int_{\R^3} g(v)(|v-Ne_1|^2+|Ne_1|^2) \dd v  \leq \frac{c}{B^2}\left(2+\frac{1}{N^2}\right)  \int_{B_1} \varphi (\tilde{v}) \dd \tilde{v} +cN^2 \leq 4.
\]
Therefore,
\[
\int_{\R^3} h(v) (1+|v|^2) \dd v \leq \int_{\R^3} (f(v)+g(v))(1+|v|^2)\dd v \leq 9,
\]
\[
\int_{\R^3}h(v) \dd v \geq \int_{\R^3}f(v) \dd v =1.
\]
We compute the entropy of $h$. Recall
\[
\int_{\R^3} f(v) |\log f(v)| \dd v \leq \int_{\R^3} f(v) (\frac{|v|^2}{2}+\frac{3}{2}\log(2\pi)) \dd v =\frac{3}{2}\log (2\pi) +\frac{3}{2}.
\]
On the other hand,
\[
\int_{\R^3} g(v)|\log g(v)| \dd v \leq c  |\log (cB^3 N)|+ c\int_{\R^3}\varphi(\tilde{v})  |\log \varphi(\tilde{v})| \dd v=\frac{ \log(B^2/N^4)}{BN^5}+c\beta_0 \leq 2+\beta_0. 
\]
Therefore,
\[
\int_{\R^3} h(v) \log h(v) \dd v  \leq \int_{\R^3} (g(v)|\log g(v)|+f(v)|\log f(v)|) \dd v \leq 10+\beta_0.
\]
\end{proof}
\begin{remark}
    Since hydrodynamic quantities of $h$ are bounded by absolute constants in Proposition \ref{prop: hydrodynamic}, the notation $Y \gtrsim Z$ means $Y\geq CZ$ for some absolute constant $C>0.$ We also emphasize that $C$ does not depend on the parameters $B$ and $N$.

\end{remark}

\subsection{The upper bound of the entropy dissipation of $h$}
We estimate the upper bound of $D(h)$
 \begin{equation}\label{e: Dissipation}
 D(h)=\frac{1}{2}\iint_{\R^3 \times \R^3} a_{ij}(v-w)h(v) h(w) \partial_i( \log h(v)-\log h(w))\partial_j ( \log h(v) -\log h (w)) \dd v \dd w.
 \end{equation}
Recall $h(v)=f(v)$ on $E_f$ and $h(v)=g(v)$ on $E_g$. Since $a_{ij}(v-w) \partial_i(\log f(v)-\log f(w))=0$, we have
\begin{align}
    D(h) \leq & \iint_{E_f \times E_g} a_{ij}(v-w)f(v)g(w) \partial_i( \log f(v)-\log g(w))\partial_j ( \log f(v) -\log g (w)) \dd v \dd w\\  
    & +\frac{1}{2} \iint_{E_g \times E_g} a_{ij}(v-w)  g(v)g(w)  \partial_i( \log g(v) -\log g(w)) \partial_j (\log g(v)- \log g(w))    \dd v \dd w\\
     \leq & \iint_{\R^3 \times E_g} a_{ij}(v-w) f(v) g(w) \partial_i (\log f(v)-\log g(w))\partial_j  (\log f(v)-\log g(w))\dd v \dd w\\
    &+\frac{1}{2}\iint_{E_g \times E_g} a_{ij}(v-w)  g(v)g(w)  \partial_i( \log g(v) - \log g(w)) \partial_j ( \log g(v)- \log g(w))    \dd v \dd w.
\end{align}
by the symmetry. Note that we used that the integrand is nonnegative in the second inequality. Since $a_{ij}$ is a positive definite matrix, it follows that
    \begin{align}
    D(h) & \leq  2 \iint_{\R^3 \times E_g} a_{ij}(v-w) f(v) g(w) \left(\partial_i \log f(v) \partial_j \log f(v)+\partial_i \log  g(w) \partial_j \log g(w) \right)\dd v \dd w\\
   & \quad +2\iint_{E_g \times E_g} a_{ij}(v-w)  g(v)g(w)  \partial_i \log g(v) \partial_j \log g(v)   \dd v \dd w.
    \end{align}
It remains to bound the following three terms:
\begin{equation} \label{e: enemy 1}
I_1:=\iint_{\R^3 \times E_g} a_{ij}(v-w)f(v)g(w) \partial_i \log f(v) \partial_j \log f(v) \dd v \dd w,
\end{equation}
\begin{equation}\label{e: enemy 2}
I_2:=\iint_{\R^3 \times E_g}   a_{ij}(v-w)f(v)g(w)  \partial_i \log  g(w) \partial_j \log g(w) \dd v \dd w,
\end{equation}
\begin{equation}\label{e: enemy 3}
I_3:=\iint_{E_g\times E_g} a_{ij}(v-w) g(v) g(w)  \partial_i \log g(v) \partial_j \log g(v)  \dd v \dd w.
\end{equation}
It turns out that $I_2$ is the dominating term, so we bound $I_1,I_3$ first.
To proceed, we need some estimates for $f(v)$ and $g(v)$.
\begin{lemma} \label{lem: 1}
For $w\in\R^3$,
\begin{equation}
 \int_{\R^3} \frac{f(v)|v|^2}{|v-w|} \dd v  \lesssim \frac{1}{|w|}.
\end{equation}
\end{lemma}
\begin{proof}
    Divide $\R^3$ into $|v-w|<\frac{|w|}{2}$ and $|v-w|\geq \frac{|w|}{2}$. Then,
    \begin{equation*}
         \int_{|v-w|<|w|/2 } \frac{f(v)|v|^2}{|v-w|} \dd v  \lesssim \exp\left(-\frac{|w|^2}{8}\right)|w|^2\int_{|v-w|<|w|/2} \frac{1}{|v-w|} \dd v \lesssim \frac{1}{|w|},
    \end{equation*}
    \begin{equation*}
         \int_{|v-w|\geq |w|/2} \frac{f(v)|v|^2}{|v-w|} \dd v  \leq  \frac{2}{|w|} \int_{\R^3} f(v)|v|^2 \dd v =\frac{6}{|w|}.
    \end{equation*}
\end{proof}
\begin{lemma} \label{lem: 2}
    For $v\in E_g$, 
    \begin{equation}
        \int_{E_g} \frac{g(w)}{|v-w|} \dd w \leq cBN=\frac{1}{N^4}.
    \end{equation}
\end{lemma}
\begin{proof}
    We have $\tilde{v} \in B_1$. It follows that 
    \[
    \int_{E_g} \frac{g(w)}{|v-w| } \dd w \leq  c\int_{B_1} \frac{\varphi(\tilde{w})}{|v-w|} \dd \tilde{w} \leq cBN\int_{B_1} \frac{\varphi(\tilde{w})}{|\tilde{v}-\tilde{w}|} \dd \tilde{w} \lesssim cBN
    \]
    since $|\tilde{v}-\tilde{w}|\leq BN |v-w|$.
\end{proof}
\begin{lemma} \label{lem: 3}
    \begin{equation}
        \int_{E_g} g(v) |\grad \log g(v)|^2 \dd v \lesssim cB^2N^2 =\frac{B}{N^3}.
    \end{equation}
   
\end{lemma}
\begin{proof}
    Recall that $\varphi(v)=\alpha_0 \exp(-(1-|v|^2)^{-1}) \mathds{1}_{|v| <1}$.
    We have
\[
|\grad \log g(v)|^2 =B^2| (N,1,1) \cdot \grad \log \varphi(\tilde{v})|^2 \lesssim B^2N^2  \frac{1}{(1-|\tilde{v}|^2)^4}.
\]
It follows that 
 \[
    \int_{E_g} g(v)|\grad \log g(v)|^2 \dd v \lesssim B^2N^2 \int_{E_g} \frac{g(v)}{(1-|\tilde{v}|^2)^4} \dd v\leq cB^2N^2 \int_{B_1} \frac{\varphi(\tilde{v}) }{(1-|\tilde{v}|^2)^4} \dd \tilde{v} \lesssim cB^2N^2
    \]
    since $\varphi(\tilde{v})(1-|\tilde{v}|^2)^{-4}$ is integrable. 
\end{proof}

\begin{prop}[Estimate for $I_1$] \label{prop: 1}
\begin{equation}
    I_1 \lesssim \frac{c}{N}=\frac{1}{BN^6}.
\end{equation}
\end{prop}
\begin{proof}
    Recall that 
    \[
a_{ij}(v) \leq \frac{1}{|v|}I.
    \]
    We replace $a_{ij}(v)$ with $\frac{1}{|v|} I$ and use Lemma \ref{lem: 1} to deduce
    \begin{equation}
        I_1 \leq \iint_{\R^3 \times E_g}  \frac{f(v)g(w)}{|v-w|}|v|^2 \dd v \dd w \lesssim \int_{E_g} \frac{g(w)}{|w|} \dd w \lesssim \frac{1}{N} \int_{E_g} g(w) \dd w \leq \frac{c}{N}.
    \end{equation}
\end{proof}
\begin{prop}[Estimate for $I_3$] \label{prop: 2}
\begin{equation}
I_3 \lesssim c^2 B^3 N^3=\frac{B}{N^7}.
\end{equation}
\end{prop}
\begin{proof}
We replace $a_{ij}(v)$ with $\frac{1}{|v|} I$ and obtain
\begin{equation}
     I_3  \leq  \iint_{ E_g \times E_g} \frac{g(v)g(w)}{|v-w|}|\grad \log g(v)|^2 \dd v \dd w .
\end{equation}
It follows from Lemmas \ref{lem: 2} and \ref{lem: 3} that
\begin{equation*}
    I_3 \leq \int_{E_g} g(v)|\grad \log g(v)|^2 \left(\int_{ E_g } \frac{g(w)}{|v-w|} \dd w \right) \dd v  \lesssim cBN \int_{E_g} g(v) |\grad \log g(v)|^2 \dd v \lesssim c^2B^3N^3.
\end{equation*}
\end{proof}
For the estimate for $I_2$, we need one more lemma that captures an anisotropy of $a_{ij}$.
\begin{lemma} \label{lem: 4}
For $w\in E_g$, 
\[
\diag(N,1,1)\left[\int_{B_{N/2}} a(w-v) f(v) \dd v\right] \diag(N,1,1) 
 \lesssim \frac{1}{N} I.
\]
\end{lemma}
\begin{proof}
    For simplicity, assume $w=Ne_1$.
    For $i=2,3$,
    \[
    e_i^t\left[\int_{B_{N/2}} a(Ne_1-v) f(v) \dd v\right]e_i \leq \int_{B_{N/2}} \frac{1}{|Ne_1-v|} f(v) \dd v \lesssim \frac{1}{N}.
    \]
    For $i=1$,
    \begin{align}
    e_1^t\left[\int_{B_{N/2}(Ne_1)} a(v) f(Ne_1-v) \dd v\right]e_1
    &= \int_{B_{N/2}(Ne_1)} \frac{v_2^2+v_3^2}{|v|^3} f(Ne_1-v) \dd v\\
    &\lesssim \frac{1}{N^3} \int_{B_{N/2}(Ne_1)}  |v-Ne_1|^2 f(Ne_1-v) \dd v\\
    & \leq \frac{1}{N^3} \int_{\R^3} |Ne_1-v|^2 f(Ne_1-v) \dd v =\frac{3}{N^3}.
    \end{align}
\end{proof}
\begin{prop}[Estimate for $I_2$] \label{prop: 3}
\begin{equation}
    I_2 \lesssim \frac{cB^2}{N}=\frac{B}{N^6}
\end{equation}
\end{prop}
\begin{proof}
    Divide the domain of $v$ into $B_{N/2}$ and its complement.
    \begin{align*}
        I_2
        &\leq\iint_{B_{N/2}^c \times E_g} |\grad \log g(w)|^2  \frac{f(v)g(w)}{|v-w|} \dd v \dd w    +  \iint_{B_{N/2} \times E_g} (\grad \log g(w))^t a(v-w) (\grad \log g(w)) f(v)g(w) \dd v \dd w
    \end{align*}
        For $w\in E_g$,
        \begin{align}
            \int_{B^c_{N/2}} \frac{f(v)}{|v-w|} \dd v
            &= \int_{B_{2N}\setminus B_{N/2}} \frac{f(v)}{|v-w|} \dd v+ \int_{B^c_{2N}} \frac{f(v)}{|v-w|} \dd v\\
            &\lesssim \exp\left(-\frac{N^2}{16}\right) \int_{B_{2N}\setminus B_{N/2}} \frac{1}{|v-w|} \dd v +\frac{1}{N} \int_{B_{2N}^c} f(v) \dd v\\
            &\lesssim N^2 \exp\left(-\frac{N^2}{16}\right) +\frac{1}{N} \exp\left(-\frac{N^2}{2}\right) \lesssim \exp\left(-\frac{N^2}{32}\right).
        \end{align}
        Applying Lemma \ref{lem: 3}, we have
        \begin{align}
        \iint_{B_{N/2}^c \times E_g} |\grad \log g(w)|^2  \frac{f(v)g(w)}{|v-w|} \dd v \dd w
        &\lesssim \exp\left(-\frac{N^2}{32}\right) \int_E  |\grad \log g(w)|^2  g(w)  \dd w
         \lesssim \exp\left(-\frac{N^2}{32}\right) cB^2N^2.
        \end{align}
   On the other hand, since
    \begin{equation}
        \grad \log g(w)=\diag(N,1,1) \grad \log \varphi(\tilde{w}),
    \end{equation}
    we obtain 
    \begin{align}
        & \iint_{B_{N/2} \times E_g} (\grad \log g(w))^t a(v-w) (\grad \log g(w)) f(v)g(w) \dd v \dd w\\
         &\lesssim  B^2\int_{E_g} \frac{1}{N}  |\grad \log \varphi(\tilde{w})|^2  g(w) \dd w\\
       & \leq  \frac{cB^2}{N} \int_{B_1} |\grad \log \varphi(\tilde{w})|^2 \varphi(\tilde{w}) \dd \tilde{w}
    \end{align}
     by Lemma \ref{lem: 4}.
    Hence, we obtain
    \[
    I_2 \lesssim \exp\left(-\frac{N^2}{32}\right) cB^2N^2 + \frac{cB^2}{N} \int_{B_1} |\grad \log \varphi(\tilde{w})|^2 \varphi(\tilde{w}) \dd \tilde{w} \lesssim \frac{cB^2}{N}
    \]
    since $|\grad \log \varphi|^2 \varphi$ is integrable. 
\end{proof}
By Propositions \ref{prop: 1}, \ref{prop: 2}, and \ref{prop: 3}, we derive the upper bound of $D(h)$. 
\begin{cor} 
    We have the following upper bound of $D(h)$ in terms of $B$ and $N$.
    \begin{equation}
    D(h) \lesssim \frac{1}{BN^6}+\frac{B}{N^7}+ \frac{B}{N^6} \lesssim \frac{B}{N^6}.
    \end{equation}
\end{cor}

\subsection{The lower bound of $h$}
We estimate the lower bound of $h$ in terms of $L^p_{-q}$ norm for $p\geq 1$ and $q\in \R$.
\begin{align*}
    \norm{h}_{L^p_{-q}}^p  \geq \int_{E} g(v)^p \Jap{v}^{-pq}\dd v \gtrsim N^{-pq} \int_{E} g(v)^p \dd v\geq \frac{(cB^3N)^p}{B^3 N} N^{-pq} \int_{B_1} \varphi(\tilde{v})^p \dd \tilde{v} \gtrsim c^p B^{3p-3}N^{p-pq-1}.
\end{align*}
Therefore,
\begin{equation} \label{e: optimal1}
\norm{h}_{L^p_{-q}} \gtrsim cB^{3-\frac{3}{p}} N^{1-q-\frac{1}{p}} =B^{2-\frac{3}{p}} N^{-4-q-\frac{1}{p}}.
\end{equation}
In particular, if $p=3$, then
\begin{equation} \label{e: optimal2}
\norm{h}_{L^3_{-q}} \gtrsim BN^{-q-\frac{13}{3}}.
\end{equation}

The following two corollaries prove Theorem \ref{thm2}.
\begin{cor}
    Let $p\geq 1$ and $q\in \R$. If there exists a constant $\kappa_1>0$ such that
    \[
    D(h) + 1 \geq \kappa_1 \norm{h}_{L^p_{-q}(\R^3)}
    \]
    holds for every solution $h$ of \eqref{e: Landau} which satisfies \eqref{e: hydronamic1}, \eqref{e: hydronamic2}, and \eqref{e: hydronamic3}, then $p\leq 3$. 
\end{cor}
\begin{proof}
     Suppose that $p>3$.
    Fix $N$ and send $B\to \infty$. It follows from \eqref{e: optimal1} that
    \[
    \frac{D(h)+1}{   \norm{h}_{L^p_{-q}(\R^3)}  } \lesssim \frac{BN^{-6}}{B^{2-\frac{3}{p}} N^{-4-q-\frac{1}{p}}}=B^{-1+\frac{3}{p}} N^{q-2+\frac{1}{p}}\to 0,
    \]
    which is a contradiction.
\end{proof}
Now we let $p=3$.
\begin{cor}
    Let $q\in \R$ be given. If there exists a constant $\kappa_2>0$ such that
    \[
    D(h) + 1 \geq \kappa_2 \norm{h}_{L^3_{-q}(\R^3)},
    \]
    for every solution $h$ of \eqref{e: Landau} which satisfies\eqref{e: hydronamic1}, \eqref{e: hydronamic2}, and \eqref{e: hydronamic3}, then $q\geq 5/3$.
\end{cor}
\begin{proof}
    Suppose that $q<5/3$.
    Fix $B=N^7$ and send $N \to \infty$. It follows from \eqref{e: optimal2} that
    \[
    \frac{D(h)+1}{   \norm{h}_{L^3_{-q}(\R^3)}  } \lesssim \frac{N}{N^{-q+\frac{8}{3}}} =N^{q-\frac{5}{3}} \to 0,
    \]
    which is a contradiction.
\end{proof}

\section*{Acknowledgments}
I would like to thank my advisor, Luis Silvestre, for helpful discussions and feedback. I would also like to thank Bill Cooperman and Elias Manuelides for kindly reading the first draft of this paper. Finally, I would like to thank the anonymous referees for their valuable comments.

\bibliographystyle{abbrv}
\bibliography{cite}

\begin{thebibliography}{10}

\bibitem{adams2003book}
R.~A. Adams and J.~J.~F. Fournier.
\newblock {\em Sobolev spaces}, volume 140 of {\em Pure and Applied Mathematics
  (Amsterdam)}.
\newblock Elsevier/Academic Press, Amsterdam, second edition, 2003.

\bibitem{alexandre2000entropy}
R.~Alexandre, L.~Desvillettes, C.~Villani, and B.~Wennberg.
\newblock Entropy dissipation and long-range interactions.
\newblock {\em Archive for rational mechanics and analysis}, 152(4):327--355,
  2000.

\bibitem{alexandre2002boltzmann}
R.~Alexandre and C.~Villani.
\newblock On the {B}oltzmann equation for long-range interactions.
\newblock {\em Comm. Pure Appl. Math.}, 55(1):30--70, 2002.

\bibitem{alexandre2004landau}
R.~Alexandre and C.~Villani.
\newblock On the {L}andau approximation in plasma physics.
\newblock {\em Ann. Inst. H. Poincar\'{e} C Anal. Non Lin\'{e}aire},
  21(1):61--95, 2004.

\bibitem{cabrera2023diffusion}
R.~Cabrera, M.~Gualdani, and N.~Guillen.
\newblock Regularization estimates of the landau-coulomb diffusion.
\newblock {\em preprint arXiv:2310.16012}, 2023.

\bibitem{calderon1961book}
A.-P. Calder\'on.
\newblock Lebesgue spaces of differentiable functions and distributions.
\newblock In {\em Proc. {S}ympos. {P}ure {M}ath., {V}ol. {IV}}, pages 33--49.
  Amer. Math. Soc., Providence, RI, 1961.

\bibitem{cameron2018global}
S.~Cameron, L.~Silvestre, and S.~Snelson.
\newblock Global a priori estimates for the inhomogeneous {L}andau equation
  with moderately soft potentials.
\newblock {\em Ann. Inst. H. Poincar\'{e} C Anal. Non Lin\'{e}aire},
  35(3):625--642, 2018.

\bibitem{chaker2020coercivity}
J.~Chaker and L.~Silvestre.
\newblock Coercivity estimates for integro-differential operators.
\newblock {\em Calc. Var. Partial Differential Equations}, 59(4):Paper No. 106,
  20, 2020.

\bibitem{chaker2022entropy}
J.~Chaker and L.~Silvestre.
\newblock Entropy dissipation estimates for the {B}oltzmann equation without
  cut-off.
\newblock {\em Kinet. Relat. Models}, 16(5):748--763, 2023.

\bibitem{desvillettes1992grazing}
L.~Desvillettes.
\newblock On asymptotics of the {B}oltzmann equation when the collisions become
  grazing.
\newblock {\em Transport Theory Statist. Phys.}, 21(3):259--276, 1992.

\bibitem{desvillettes2015entropy}
L.~Desvillettes.
\newblock Entropy dissipation estimates for the {L}andau equation in the
  {C}oulomb case and applications.
\newblock {\em J. Funct. Anal.}, 269(5):1359--1403, 2015.

\bibitem{desvillettes2005boltzmann}
L.~Desvillettes and C.~Villani.
\newblock On the trend to global equilibrium for spatially inhomogeneous
  kinetic systems: the {B}oltzmann equation.
\newblock {\em Invent. Math.}, 159(2):245--316, 2005.

\bibitem{golse2022partial}
F.~Golse, M.~P. Gualdani, C.~Imbert, and A.~Vasseur.
\newblock Partial regularity in time for the space-homogeneous {L}andau
  equation with {C}oulomb potential.
\newblock {\em Ann. Sci. \'{E}c. Norm. Sup\'{e}r. (4)}, 55(6):1575--1611, 2022.

\bibitem{golse2022partialarxiv}
F.~Golse, C.~Imbert, S.~Ji, and A.~F. Vasseur.
\newblock Local regularity for the space-homogeneous landau equation with very
  soft potentials.
\newblock {\em preprint arXiv:2206.05155}, 2024.

\bibitem{goudon1997fokker}
T.~Goudon.
\newblock On {B}oltzmann equations and {F}okker-{P}lanck asymptotics: influence
  of grazing collisions.
\newblock {\em J. Statist. Phys.}, 89(3-4):751--776, 1997.

\bibitem{gressman2011sharp}
P.~T. Gressman and R.~M. Strain.
\newblock Sharp anisotropic estimates for the boltzmann collision operator and
  its entropy production.
\newblock {\em Advances in Mathematics}, 227(6):2349--2384, 2011.

\bibitem{guillen2023global}
N.~Guillen and L.~Silvestre.
\newblock The landau equation does not blow up.
\newblock {\em preprint arXiv:2311.09420}, 2023.

\bibitem{imbert2021regularity}
C.~Imbert and L.~Silvestre.
\newblock Regularity for the boltzmann equation conditional to macroscopic
  bounds.
\newblock {\em EMS Surveys in Mathematical Sciences}, 7(1):117--172, 2021.

\bibitem{imbert2022global}
C.~Imbert and L.~E. Silvestre.
\newblock Global regularity estimates for the {B}oltzmann equation without
  cut-off.
\newblock {\em J. Amer. Math. Soc.}, 35(3):625--703, 2022.

\bibitem{jones1981sobolev}
P.~W. Jones.
\newblock Quasiconformal mappings and extendability of functions in {S}obolev
  spaces.
\newblock {\em Acta Math.}, 147(1-2):71--88, 1981.

\bibitem{mouhot2006linear}
C.~Mouhot.
\newblock Explicit coercivity estimates for the linearized {B}oltzmann and
  {L}andau operators.
\newblock {\em Comm. Partial Differential Equations}, 31(7-9):1321--1348, 2006.

\bibitem{luke2006sobolev}
L.~G. Rogers.
\newblock Degree-independent {S}obolev extension on locally uniform domains.
\newblock {\em J. Funct. Anal.}, 235(2):619--665, 2006.

\bibitem{silvestre2016new}
L.~Silvestre.
\newblock A new regularization mechanism for the {B}oltzmann equation without
  cut-off.
\newblock {\em Comm. Math. Phys.}, 348(1):69--100, 2016.

\bibitem{silvestre2017upper}
L.~Silvestre.
\newblock Upper bounds for parabolic equations and the landau equation.
\newblock {\em Journal of Differential Equations}, 262(3):3034--3055, 2017.

\bibitem{silvestre2022regularity}
L.~Silvestre.
\newblock Regularity estimates and open problems in kinetic equations.
\newblock {\em preprint arXiv:2204.06401}, 2022.

\bibitem{stein1970book}
E.~M. Stein.
\newblock Singular integrals, harmonic functions, and differentiability
  properties of functions of several variables.
\newblock In {\em Singular {I}ntegrals ({P}roc. {S}ympos. {P}ure {M}ath.,
  {C}hicago, {I}ll., 1966)}, pages 316--335. Amer. Math. Soc., Providence, RI,
  1967.

\bibitem{villani1998new}
C.~Villani.
\newblock On a new class of weak solutions to the spatially homogeneous
  {B}oltzmann and {L}andau equations.
\newblock {\em Arch. Rational Mech. Anal.}, 143(3):273--307, 1998.

\bibitem{villani1999regularity}
C.~Villani.
\newblock Regularity estimates via the entropy dissipation for the spatially
  homogeneous {B}oltzmann equation without cut-off.
\newblock {\em Rev. Mat. Iberoamericana}, 15(2):335--352, 1999.

\end{thebibliography}

\end{document}